\documentclass[12pt]{amsart}

\usepackage{amsmath,amssymb,amsthm,amscd,amsfonts}
\usepackage{verbatim}
\usepackage{comment}
\usepackage{multirow}
\usepackage{mathtools}
\usepackage{enumitem}
\usepackage{pgf,tikz}
\usepackage{tikz-cd}
\usetikzlibrary{arrows}
\usetikzlibrary{cd} 
\usepackage[normalem]{ulem}
\usepackage{graphicx}
\usepackage{wasysym}
\usepackage{csquotes}
\usepackage{bbm}
\usepackage{bm}
\usepackage{color}
\usepackage{tabularx}
\usepackage{stmaryrd}
\usepackage{pdfpages}
\usepackage{forest}
\usepackage{xcolor}

\usepackage[english]{babel}
\usepackage[margin=2.5cm]{geometry}
\usepackage{enumitem}

\numberwithin{equation}{section}
\usepackage[tickmarkheight=.5em,colorinlistoftodos,textsize=tiny]{todonotes}
\usepackage{xargs}

\newcommand{\PP}{\mathbb{P}}

\theoremstyle{plain}

\newtheorem{thm}{Theorem}[section]

\newtheorem{prop}[thm]{Proposition}
\newtheorem{lemma}[thm]{Lemma}
\newtheorem{question}[thm]{Question}
\theoremstyle{definition}
\newtheorem{df}[thm]{Definition}
\newtheorem{ex}[thm]{Example}
\theoremstyle{remark}
\newtheorem{rmk}[thm]{Remark}

\definecolor{rossred}{rgb}{1.0,0.25,0.66}  
\definecolor{rossgreen}{rgb}{0.25,0.66,0.25} 
\definecolor{rossblue}{rgb}{0.25,0.66,1.0}
\definecolor{sashapurple}{rgb}{0.5,0.15,0.5}

\usepackage[pdfborder={0 0 0}]{hyperref}
\hypersetup{
    colorlinks=true,
    linkcolor=rossblue,
    citecolor=rossgreen
}
\usepackage{url}
\usepackage{enumitem}
\usepackage[utf8]{inputenc}
\usepackage[T1]{fontenc}

\usepackage{tikz}
\usetikzlibrary{arrows,shapes,positioning}
\tikzstyle{every loop}= []

\colorlet{myGray}{gray!25}

\tikzset{my circle/.style={circle,draw=black,fill=myGray,inner
    sep=0pt,minimum size=6pt}} \tikzset{my square/.style={regular
    polygon,regular polygon sides=4,draw=black,fill=myGray,inner
    sep=0pt,minimum size=9pt}} \tikzset{my star/.style={star,star
    point ratio=2.5,draw=black,fill=myGray,inner sep=0pt,minimum
    size=9pt}} \tikzset{my triangle/.style={regular polygon,regular
    polygon sides=3,draw=black,fill=myGray,inner sep=0pt,minimum
    size=9pt}} \tikzset{my
  kite/.style={diamond,aspect=0.25,draw=black,fill=myGray,inner
    sep=1pt,minimum size=6pt}}
\tikzset{every pin/.style={pin distance=3pt,inner sep=1pt,font=\tiny}}

\numberwithin{equation}{section}
\title{Geometrically vertex decomposable star configurations}

\author[Cooper]{Susan M. Cooper}
\address[S. M. Cooper]
{Department of Mathematics, University of Manitoba, Winnipeg, MB R3T 2M6, Canada}
\email{susan.cooper@umanitoba.ca}

\author[Marangone]{Emanuela Marangone}
\address[E. Marangone]
{Department of Mathematics, University of Manitoba, Winnipeg, MB R3T 2M6, Canada}
\email{emanuela.marangone@umanitoba.ca}

\author[Guardo]{Elena Guardo}
\address[E. Guardo]
{Dipartimento di Matematica e Informatica, Università di Catania, 95125, Italy}
\email{guardo@dmi.unict.it}

\author[Van Tuyl]{Adam Van Tuyl}
\address[A. Van Tuyl]
{Department of Mathematics and Statistics
McMaster University, Hamilton, ON L8S 4L8, Canada}
\email{vantuyla@mcmaster.ca}

\keywords{geometrically vertex decomposable, 
star configuration, Knutson ideal}
\subjclass[2020]{14N20, 14J70, 13P10}

\date{\today}
\begin{document}

\begin{abstract}
The goal of this paper is to determine how the family
of  ideals of  star configurations intersects with the 
class of geometrically vertex decomposable ideals.   
The main result of this paper shows that
the answer is subtle since the geometrically vertex decomposability property of an ideal is {\it not} invariant under a linear change of 
variables, and thus the answer will depend upon the choice of the
linear forms that define the ideal of the star configuration.
We also show that the ideal of a star  configuration is a 
Knutson ideal precisely when it is a 
geometrically vertex decomposable  ideal.
\end{abstract}
\maketitle

\section{Introduction}

An ideal of a {\it star configuration} of codimension 
$1 \leq c \leq n$ in 
$\mathbb{P}^n$ is constructed from a collection of $\ell \geq c$ linear 
forms $\mathbb{L} = \{L_1,\ldots,L_\ell\}$ under some suitable hypotheses by 
taking the intersection of all ideals generated by
all subsets of $\mathbb{L}$ of size $c$ (see  
Definition \ref{defn:starconfig} for complete details).  We denote
the variety defined by this ideal by $\mathbb{X}(\ell,c)$
and its corresponding ideal as $I_{\mathbb{X}(\ell,c)}$.
For the special case $\mathbb{X}(5,2) \subseteq \mathbb{P}^2$, the variety is the ten points in the plane given by the intersection of five general lines, whose defining linear forms are  $\{L_1,\ldots,L_5\}$. In this 
situation, as shown in Figure \ref{starconfigpicture},
the five lines resemble
a star, thus providing the inspiration for the
name.
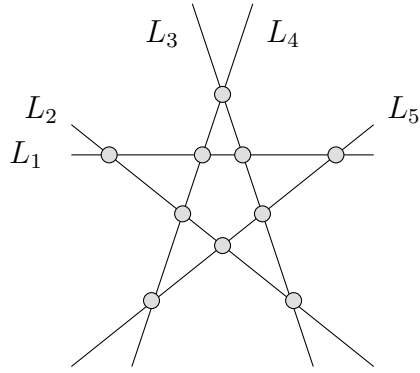
\begin{figure}[!ht]
  \centering
    \begin{tikzpicture}[scale=0.4]
      \draw (-5,3)--(5,3);
      \draw (1,8)--(-3,-4);
      \draw (-1,8)--(3,-4);
      \draw (5,4)--(-5,-4);
      \draw (-5,4)--(5,-4);
      \node at (-6.5,3) {$L_1$};
      \node at (-6,4.5) {$L_2$};
      \node at (-2,7) {$L_3$};
      \node at (2,7)  {$L_4$};
      \node at (6,4.5) {$L_5$};
      \node[my circle] at (0,5) {};
      \node[my circle] at (-.666,3) {};
      \node[my circle] at (.666,3) {};
      \node[my circle] at (3.75,3) {};
      \node[my circle] at (-3.75,3) {};
      \node[my circle] at (0,0) {};
      \node[my circle] at (-2.35,-1.82) {};
       \node[my circle] at (2.35,-1.82) {};
      \node[my circle] at (-1.32,1.05) {};
      \node[my circle] at (1.32,1.05) {};
    \end{tikzpicture}
    \caption{$\mathbb{X}(5,2)$, a star configuration of 10 points in $\mathbb{P}^2$}
    \label{starconfigpicture}
\end{figure}
While the name and many of the
foundational properties of these varieties can be  found in
\cite{GHM13}, they appear earlier in the 
literature in \cite{CHT11, Dbook, GMS2006,OS10}, where they are sometimes called ``$\ell$-laterals''.  
Star configurations of the form $\mathbb{X}(\ell,c) \subseteq \mathbb{P}^n$ have many nice geometric, algebraic,
and homological properties which can determined
directly from the values of $\ell, c,$ and $n$; see \cite{BH2010,CGVT14,  CVT2011, CDG2022, M2020}.

Klein and Rajchgot \cite{KR21} introduced the notion
of a {\it geometrically vertex decomposable ideal} (see
Definition \ref{dfGVD}), based upon the geometric vertex decomposition first described by
Knutson, Miller, and Yong \cite{KMY}.  Geometrically vertex decomposable
ideals can be viewed as a generalization of Stanley-Reisner
ideals of vertex decomposable simplicial complexes since they
share many of the same properties, including similar
liaison theoretic  properties \cite{KR21}. The last five years have seen
a number of papers add to our understanding  of 
geometrically vertex decomposable ideals,
including connections with toric ideals of graphs
\cite{BVTZ, CDSRVT}, matrix Schubert varieties
\cite{K23}, homological invariants like Castelnuovo-Mumford regularity
\cite{NRVT24}, Gorenstein liaison \cite{FKRS},
and Frobenius splittings \cite{DSH,DGKRS}.

Given the importance of both classes of ideals in commutative algebra and
algebraic geometry, it is natural to ask when the ideal of a
star configuration is also geometrically vertex decomposable. One might initially expect that the property of
being a geometrically vertex decomposable ideal would only depend
upon the parameters $\ell,c$, and $n$.  However, as the following 
example illustrates, the answer is more subtle since it depends upon on
$\mathbb{L}$, the linear forms that define
$\mathbb{X}(\ell,c)$.    This example will be used to illuminate
the statement of our main theorem.

\begin{ex}\label{ex.intro}
Let $R = k[x_0,x_1,x_2]$ denote the coordinate ring of $\mathbb{P}^2$.
    In $\mathbb{P}^2$, any three non-collinear points are an example
    of a star configuration $\mathbb{X}(3,2)$.  Consider
    the following two such star configurations:
    \begin{eqnarray*}
        \mathbb{X}_1 & = & \{[1:0:0], [0:1:0], [0:0:1]\}, ~\mbox{and}  \\
        \mathbb{X}_2 & = & \{[1:1:326],[1:848:1], [95125:1:1]\}\footnotemark. 
    \end{eqnarray*}
    \footnotetext{The set $\mathbb{X}_2$ was constructed from the postal codes of the authors:  R3T 2M6, L8S 4L8, and 95125.}
    The set $\mathbb{X}_1$ is the star configuration 
    $\mathbb{X}(3,2)$
    constructed from $\mathbb{L}_1 = \{x_0,x_1,x_2\}$, while the second
    set is the star configuration $\mathbb{X}(3,2)$ constructed from 
    \footnotesize
    \begin{multline*}
        \mathbb{L}_2 = \{ 276\,447\,x_0-325\,x_1-847\,x_2,~  
        325\,x_0-31\,010\,749\,x_1+95\,124\,x_2,~        
        847\,x_0+95\,124\,x_1-80\,665\,999\,x_2\}
    \end{multline*}
    \normalsize
    The defining ideal of $\mathbb{X}_1$ is $I_{\mathbb{X}_1} =
    (x_0x_1,x_0x_2,x_1x_2)$.  This is a squarefree monomial ideal whose
    associated simplicial complex via the Stanley-Reisner
    correspondence is a vertex decomposable simplicial complex.
    Consequently, by \cite[Proposition 2.9]{KR21}, the ideal
    $I_{\mathbb{X}_1}$ is geometrically vertex decomposable.

    On the other hand, the ideal $I_{\mathbb{X}_2}$ is not
    geometrically vertex decomposable.  This can be checked
    using the \emph{Macaulay2} package {\tt GeometricDecomposability} \cite{CVT24}.  An
    alternative strategy is to exploit a property
    of the Gr\"obner basis of geometrically vertex
    decomposable ideals, as first described in
    \cite[Lemma 2.6]{KR21} (also see Lemma \ref{lemma:squarefree}).

    We point out that
    the difference is {\it not} a result of $I_{\mathbb{X}_1}$ being
    a monomial ideal.  If $\mathbb{L}_3 = \{x_2,x_1+x_2,x_0+x_1+x_2\}$,
    then the resulting star configuration is
    $$\mathbb{X}_3 = \mathbb{X}(3,2) = \{[1:0:0],[1:-1:0],[0:1:-1]\}.$$
    In this case, the ideal of the star configuration is
    $$I_{\mathbb{X}_3} = \left(x_1\,x_2+x_2^{2},\,x_0\,x_2+x_1\,x_2+x_2^{2},\,x_0\,x_1+x_1^{2}+x_0\,x_2+2\,x_1\,x_2+x_2^{2}\right).$$
    This ideal is not a monomial ideal, but it is geometrically vertex decomposable as shown in Example \ref{exGVD}.  
\end{ex}

Since any three non-collinear points in $\mathbb{P}^2$ can be 
    mapped to any other three non-collinear points by a linear 
    change of variables, the above example highlights that the 
    property of being a geometrically vertex decomposable ideal is not
    a geometric invariant, but 
    an algebraic property that depends
    upon the choice of the forms in $\mathbb{L}$.  To capture the
    needed information, we associate to any collection of linear forms
     $\mathbb{L} = \{L_1,\ldots,L_\ell\}$ in
    $R = k[x_0,\ldots,x_n]$ an $\ell \times (n+1)$ matrix where
    the $(i,j)$-th entry is the coefficient of $x_{j+1}$ in the form
    $L_i$.  For the three examples given in Example \ref{ex.intro},
    the corresponding matrices are
    $$\begin{bmatrix}
    1 & 0 & 0 \\
    0 & 1 & 0 \\
    0 & 0 & 1
        \end{bmatrix},
        \begin{bmatrix}
    276\, 447 & -325 & -847 \\
    325 & -31\, 0310\, 749  & 95\, 124 \\
    847 & 95\, 124 & -80\, 665\, 999
        \end{bmatrix}, ~\mbox{and}~
\begin{bmatrix}
    0 & 0 & 1 \\
    0 & 1 & 1 \\
    1 & 1 & 1
        \end{bmatrix}.
 $$
Note that for the two ideals that are geometrically vertex decomposable,
the associated matrix constructed from $\mathbb{L}$ is triangular, but 
this is not the case for the other ideal. 

This
observation about the ``shape'' of 
the matrix turns out to be the key to classify the ideals of
star configurations that are geometrically vertex decomposable.
Our main result refines this 
observation.

\begin{thm}\label{mainthm} Let $\mathbb X(\ell,c)$ be a star 
configuration of codimension $1 \leq c \leq n$ in $\mathbb{P}^n$ 
defined by the general linear forms 
$\mathbb L = \{L_1, \ldots, L_{\ell}\}$ with $\ell\geq c$.  
\begin{enumerate}
    \item If $\ell = c$, then $I_{\mathbb{X}(\ell,c)}$ is always
    a geometrically vertex decomposable ideal.
    \item If $\ell =c+1$, then $I_{\mathbb{X}(\ell,c)}$ is a geometrically
    vertex decomposable ideal if and only if (up to reordering the variables and the linear forms in $\mathbb L$) the coefficients of the linear forms in $\mathbb L$ form a matrix that has a $(c+1)\times(c+1)$ triangular submatrix
$$
\left[ \begin{array}{ccccccc}
0 & 0 & \cdots & a_{1,c}&*& \cdots &* \\
\vdots & \vdots & \vdots & \vdots&\vdots& \vdots &\vdots  \\
0 & a_{c,1} &  \cdots & a_{c,c}  & *& \cdots &* \\
a_{c+1,0}& a_{c+1,1}& \cdots   & a_{c+1,c}&*& \cdots &* 
\end{array}
\right].$$
\item If $c+1<\ell \leq n+1$, then $I_{\mathbb{X}(\ell,c)}$ is a 
    geometrically vertex decomposable ideal if and only if 
    $I_{\mathbb{X}(\ell,\ell-1)}$ is a geometrically vertex
    decomposable ideal.   
    %the ideal of the star configuration of codimension $\ell-1$ defined by $\mathbb L$ is geometrically vertex decomposable.
   \item If $\ell > n+1$, then $I_{\mathbb{X}(\ell,c)}$ is never a
    geometrically vertex decomposable ideal. 
\end{enumerate}
\end{thm}
\noindent
To prove  Theorem \ref{mainthm} (2) and (3), we rely on 
a result of Klein-Rajchgot \cite{KR21} that a 
geometrically vertex decomposable
ideal must have a generating set that is squarefree in a variable.  
Theorem \ref{mainthm} (4) requires the $h$-vector
of star configurations, as  first described by 
Geramita-Harbourne-Migliore \cite{GHM13}, and
a property of the 
Hilbert functions of geometrically vertex decomposable
ideals due to  Nguy$\tilde{\text{\^e}}$n-Rajchgot-Van Tuyl
\cite{NRVT24}. 

In addition, we also exhibit a link 
between geometrically vertex decomposable ideals and 
{\it Knutson ideals} (see Definition \ref{kn}).
These ideals first arose in  Knutson's work
related to Frobenius splittings \cite{K2009}.
The name for these ideals comes from Conca-Varbarro \cite{CV2020}, where
some of their properties were described.  Seccia later
described some families of Knutson ideals \cite{S2021,S2022}.
As recently shown by De Negri, Gorla, Klein, Rajchgot, 
and Seccia \cite{DGKRS}, there appears to be a relationship
between Frobenius splittings, Knutson ideals, and 
geometrically vertex decomposable ideals.  
While the set of Knutson ideals and geometrically
vertex decomposable ideals are not the same, in the case
of ideals of star configurations, we are able to make
the following contribution.

\begin{thm}[Theorem \ref{mainthm2}]
Let $\mathbb{X}(\ell,c)$ be a star configuration of 
codimension $c$ in $\mathbb{P}^n$.
    Then $I_{\mathbb{X}(\ell,c)}$ 
    is a geometrically vertex decomposable ideal if and only if 
    $I_{\mathbb{X}(\ell,c)}$ is a Knutson ideal.
\end{thm}

Our paper is structured as follows.  In Section 
\ref{sec.backgroud} we provide the needed background
on geometrically vertex decomposable ideals.
In Section \ref{sec.starconfig} we introduce star
configurations, and prove Theorem \ref{mainthm} (4).
The rest of Theorem \ref{mainthm} is proved in
Section \ref{sec.classification}.  In Section \ref{sec.knutson} we show that the
ideal of a star configuration is Knutson if and only
if it is geometrically vertex decomposable.
Finally, Section \ref{sec.future} considers future directions
and takes the first step towards extending our results
to generalized star configurations.

%%%%%%%%%%%%%%%%%%%%%%%%%%%%%%%%%%%%%%%%%%%%%%%%%%%%%%%%%%%

\section{Background}\label{sec.backgroud}

We let $k$ be an infinite field and $R$ be the polynomial ring $R = k[x_0, \ldots, x_n]$. For some $0 \leq j \leq n$, fix $y = x_j$ and set $S = k[x_0, \ldots, \widehat{x_j}, \ldots, x_n]$.  Observe that if $f \in R$, we can write $f = \sum_{i=0}^d \alpha_i y^i$, where $\alpha_i \in S$.  If $0 \not = f \in R$ and $<$ is a monomial order on $R$, then we call $\alpha_d y^d$ the \emph{initial $y$-form} of $f$.  Extending to ideals $J \subset R$, we 
let ${\rm in}_y(J) = ({\rm in}_y (f) : f \in J)$.  The monomial order $<$ is called \emph{$y$-compatible} if ${\rm in}_<({\rm in}_y(f)) = {\rm in}_<(f)$ for all $f \in R$.
\begin{rmk}
    Given a $y$-compatible order $<$ for an ideal $I \subset R = k[x_0, \ldots, x_n]$, if $\mathcal G = \{g_1, \ldots, g_m\}$ is a Gr\"obner basis of $I$, then we write $g_i = y^{d_i}q_i + r_i$, with ${\rm in}_y(g_i) = y^{d_i}q_i$ and where $y$ does not divide any term of $q_i$.  With this notation, we define the ideals
$$C_{y,I} = (q_1, \ldots, q_m) \,\,\,\,\, \text{and} \,\,\,\,\, N_{y,I} = (q_i \,:\, d_i = 0).$$
\end{rmk}
Since $x_j$ does not appear in the generators of $C_{y,I}$ and 
$N_{y,I}$, with abuse of notation we let $C_{y,I}$ and $N_{y,I}$
denote both an ideal in $R=k[x_0,, \ldots, x_n]$ and in 
$S = k[x_0, \ldots, \widehat{x_j}, \ldots, x_n]$.

\begin{df}
Let $I \subset R = k[x_0, \ldots, x_n]$ be an ideal and $<$ be a $y$-compatible monomial order on $R$.  When
$${\rm in}_y(I) = C_{y,I} \cap (N_{y, I} + (y)),$$
we call this decomposition a \emph{geometric vertex decomposition} of $I$ with respect to $y$.
\end{df}

From Klein and Rajchgot \cite{KR21} we have that:
\begin{lemma}\label{lemma: indep.}
Let $I \subset R = k[x_0, \ldots, x_n]$ and let $y=x_j$. The ideals $C_{y,I}$ and $N_{y,I}$, as in the above definition, are independent of the choice of  Gr\"obner basis and of the $y$-compatible monomial order $<$. They are uniquely determined as ideals in $S=k[x_0, \ldots, \widehat{x_j}, \ldots, x_n]$ by
$$C_{y,I}=({\rm in}_y(I) :y^{\infty}) \quad \text{and} \quad N_{y,I}+(y)= {\rm in}_y(I) +(y). $$
\end{lemma}

\begin{df} \label{dfGVD}
An ideal $I \subset R = k[x_0, \ldots, x_n]$ is said to be \emph{geometrically vertex decomposable (GVD)} if $I$ is unmixed and
\begin{enumerate}
  \item[(a)] $I = (1)$, or $I$ is generated by a (possibly empty) subset of $\{x_0, \ldots, x_n\}$, or
  \item[(b)] there exists $y = x_j$ such that we have a geometric vertex decomposition ${\rm in}_y(I) = C_{y,I} \cap (N_{y, I} + (y))$, and the contractions of $C_{y,I}$ and $N_{y,I}$ to $S=k[x_0, \ldots, \widehat{x_j}, \ldots, x_n]$ are geometrically vertex
  decomposable.
  \end{enumerate}
\end{df}

Recall that we say that a polynomial $f \in R$ is {\it squarefree in 
a variable $x_j$} if $x_j^2$ does not divide any term in the 
support of $f$.  An ideal $I$ is {\it squarefree in $x_j$} if there
is a generating set of $I$ such that every element in the
generating set is squarefree in $x_j$.

\begin{lemma}[{\cite[Lemma 2.6]{KR21}}]\label{lemma:squarefree}
Let $I \subset R = k[x_0, \ldots, x_n]$ be an ideal,
and let $<$ be a $y$-compatible monomial order on $R$.  If $I$ has a geometric vertex decomposition with respect to $y = x_j$, then $I$ is squarefree in $y$, and the reduced Gr\"obner basis of $I$ with respect to any $y$-compatible term order has the form
  $$\{yq_1+r_1, \ldots, yq_{\ell}+r_{\ell}, h_1, \ldots, h_t\}$$
  where $y$ does not divide any term of any $q_i$ or $r_i$ for any $1 \leq i \leq \ell$ nor any $h_j$ for $1 \leq j \leq t$.
\end{lemma}

By Lemma~\ref{lemma:squarefree}, a necessary condition for an ideal to be geometrically vertex decomposable is the existence of a variable $y$ such that the ideal admits a Gröbner basis that is squarefree in $y$. The following lemma will be used to show that an ideal is not 
geometrically vertex decomposable:

\begin{lemma}\label{mindeg}
Let $I$ be a homogeneous ideal of $R$, and let $d$ be the minimal degree of a nonzero element of $I$.
If there is an element of $I$ of degree $d$ that is not squarefree with respect to the variable $x_i$, then there does not exist a Gröbner basis 
for $I$ that is squarefree in $x_i$.
\end{lemma}

\begin{proof}
    Fix any monomial order $<$, and let $\{b_1,\dots, b_m\}$ be a 
    Gröbner basis for the ideal $I$ with respect to this order.  
    Because $I$ is a homogeneous ideal, 
    we can assume that all the $b_i$'s are 
    homogeneous.
    Since the minimal degree of an element of $I$ is $d$, 
    we must have $\deg(b_j) \geq d$ for $1 \leq i \leq m$. 
    
    Let $f\in I$ be an element of degree $d$ that is not squarefree
    with respect to the variable $x_i$. Since $\{b_1,\dots, b_m\}$ generates $I$, we can write $f=\sum a_jb_j$. By the minimality of the degree $d$, the $a_i$'s must be constants or zero. Because $f$ is not squarefree in $x_i$, $x_i^2$ must appear in at least one of the $b_j$. We conclude that $\{b_1,\dots, b_m\}$ is not squarefree in $x_i$.
\end{proof}

The following lemma, which is of independent
interest, provides a large class
of ideals that are geometrically vertex decomposable
and will be used to prove Theorem \ref{mainthm} (1).

\begin{lemma}\label{CI}
     Let $I\subseteq R = k[x_0,\dots,x_n]$ be a complete intersection generated by linear forms. Then $I$ is geometrically vertex
     decomposable.
\end{lemma}
\begin{proof}
Assume $I=(L_1,\dots, L_c)$ is a complete intersection generated by
linear forms.  Then $1\leq c\leq n+1$.  Note that because
$I$ is a complete intersection, this forces
$L_1,\ldots,L_c$ to be linearly independent.
We proceed by induction on $n$. The case $n=0$ is clear.

If there is a variable $x_i$ that does not appear in any
of the $L_j$'s, then $I$ can be viewed as an ideal
of $S = k[x_0,\ldots,\hat{x}_i,\ldots,x_n]$.  By induction,
$I$ is a geometrically vertex decomposable
ideal of $S$.  But then $R/I \cong S/I \otimes k[x_i]$,
and the result now follows from \cite[Theorem 2.9]{CDSRVT}.
Hence, we can assume each variable appears in at least one
$L_i$.

Because $L_1,\ldots,L_c$ are linear forms, the algorithm to
find the Gr\"obner basis of $(L_1,\ldots,L_c)$ is equivalent to
applying the row reduction algorithm from linear algebra to 
the matrix of coefficients of $L_1,\dots,L_c$.  That is, we obtain new linear forms $L'_1,\dots,L'_c$ such that $I=(L'_1,\dots,L'_c)$. After possibly rearranging the variables and the order of 
the $L'_i$'s, the matrix of the coefficients 
of $\{L'_1,\ldots,L'_c\}$ has the form
\begin{align*}
 &\begin{array}{ccccccc} \quad \ \scriptstyle{x_0}  &  \quad \scriptstyle{x_1} \quad  & \ \ \cdots & \quad  \scriptstyle{x_{c-1}} & \quad \scriptstyle{x_{c}} & \cdots &  \scriptstyle{x_{n}} \end{array}\\
    \begin{array}{c}\scriptstyle{L'_1}\\\vdots\\ \scriptstyle{L'_{c-1}}\\\scriptstyle{L'_{c}}\end{array}&
\left[ \begin{array}{ccccccc}
0 & 0 & \cdots & a_{1,c-1}&*& \ \cdots \ &* \\
\vdots & \vdots & \vdots & \vdots&\vdots& \ \vdots \ &\vdots  \\
0 & a_{c-1,1} &  \cdots & a_{c-1,c-1}  & *& \ \cdots \ &* \\
a_{c,0}& a_{c,1}& \cdots   & a_{c,c-1}&*&\  \cdots \ &* 
\end{array}
\right].\end{align*}

%If $a_{c,0}=0$, we can view $I=(L'_1,\dots, L'_c)$ as an ideal in $ k[x_1,\dots, x_n]$, and by inductive hypothesis is GVD. 
We may assume $a_{c,0}\neq 0$, since $x_0$ must appear in at least one linear form.  We set $y=x_0$ and
let $<$ be the lexicographical order with $y>x_1>\cdots$,
which is also $y$-compatible.
Then $\{L'_1,\dots, L'_c\}$  is already a Gröbner basis for $I$ with respect to this order and $y$ appears only in $L'_c$. So we obtain the geometric vertex decomposition
$$C_{y,I} = (L'_1,\ldots,L'_{c-1},a_{c,0})=(1) \,\,\,\,\, \text{and} \,\,\,\,\, N_{y,I} = (L'_1,\ldots,L'_{c-1}).$$
By the induction hypothesis, $N_{y,I}$  is 
geometrically vertex decomposable, while $C_{y,I}$ is 
geometrically vertex decomposable by definition. Therefore, we conclude that $I$ is geometrically vertex
decomposable.
\end{proof}

Since $R$ is an $\mathbb{N}$-graded ring, $R = \bigoplus_{t \in \mathbb{N}} R_t$ where $R_t$ denotes the vector space of
homogenous forms of degree $t$.
When $I \subseteq R$ is a homogeneous ideal, observe that the homogeneous pieces $I_t$ are each a finite dimensional vector space over $k$ and subspaces of $R_t$.  We record the degree-by-degree dimensions of the homogeneous pieces using a special function.

\begin{df}
If $I \subseteq R = k[x_0, \ldots, x_n]$ is a homogeneous ideal, then the \emph{Hilbert function} of $R/I$ is the numerical function $HF_{R/I}(-):  \mathbb{N} \rightarrow \mathbb{N}$ where
$$HF_{R/I}(t) = \dim_k (R/I)_t = \dim_k R_t - \dim_k I_t = \binom{n+t}{t} - \dim_k I_t.$$
\end{df}

We can record the Hilbert function as a sequence $$HF_{R/I} = (HF_{R/I}(0), HF_{R/I}(1), HF_{R/I}(2), \ldots) = (h_0, h_1, h_2, \ldots).$$  Doing so, we define the \emph{first difference Hilbert function} by $\Delta HF_{R/I} = (h_0, h_1-h_0, h_2-h_1, \ldots)$.  We then recursively define the \emph{second}, \emph{third}, etc. differences.

\begin{df}
For a variety $V\subseteq \mathbb{P}^n$ of codimension $c$, the {\em h-vector} of $V$ (or of $R/I_V$)
is the $(n - c+1)$-st difference of the Hilbert function of $R/I_V$.
\end{df}

If $I \subseteq R = k[x_0, \ldots, x_n]$ is a homogeneous ideal,  then there exists a unique polynomial $HP_{R/I}(t) \in \mathbb Q[t]$ for which $HF_{R/I}(t) = HP_{R/I}(t)$ for all $t \gg 0$.  We call the polynomial $HP_{R/I}(t)$ the \emph{Hilbert polynomial} of $R/I$.

\begin{df}\label{Hilbertian}
    Let $I \subseteq R = k[x_0,\ldots,x_n]$ and let  $HP_{R/I}(t)$  and $HF_{R/I}(t)$ be the Hilbert polynomial and the Hilbert function of $R/I$, respectively. We say that
    \begin{itemize}
        \item[$(a)$]  $R/I$ is \emph{Hilbertian}  if $HP_{R/I}(t) = HF_{R/I}(t)$ for all $t \geq 0$;
        \item[$(b)$]  $R/I$ is \emph{almost Hilbertian} $HP_{R/I}(t) = HF_{R/I}(t)$ for all $t \geq 1$.
    \end{itemize}
\end{df}

These properties are connected with the geometrically vertex decomposability of the ideal as shown in \cite{NRVT24}:
\begin{thm}{\cite[Corollary 5.8]{NRVT24}}\label{hilbertian}
    Suppose that $I \subseteq R = k[x_0,\ldots,x_n]$ is a homogeneous
    ideal.  If $I$ is a geometrically vertex decomposable ideal,
    then $R/I$ is almost Hilbertian.
\end{thm}

We can now prove that if $I_{\mathbb{X}}$ is an ideal 
of {\it any} set $\mathbb{X}$ of  more than $n+1$ points 
in $\mathbb{P}^n$, then 
$I_{\mathbb{X}}$ cannot be a geometrically vertex decomposable ideal.
\begin{thm}\label{points}
    Let $\mathbb{X} \subseteq \mathbb{P}^n$ be any set of
    points.  If $|\mathbb{X}| > n+1$, then $I_{\mathbb{X}}$ is not
    geometrically vertex decomposable.
\end{thm}

\begin{proof}
    It is well known that the Hilbert polynomial of $R/I_{\mathbb{X}}$ 
    for any set of points is $HP_{R/I_\mathbb{X}}(t) = |\mathbb{X}|$, i.e.,
    the constant polynomial.  
    However, since $I_\mathbb{X} \subseteq R = k[x_0,\ldots,x_n]$, we have 
    $$HF_{R/I}(1) \leq \dim_k R_1 = n+1 < |\mathbb{X}|.$$
    The conclusion now follows from Theorem \ref{hilbertian}.
\end{proof}

%%%%%%%%%%%%%%%%%%%%%%%%%%%%%%%%%%%%%%%%%%%%%%%%%%%%%%%

\section{Star Configurations in $\PP^n$}\label{sec.starconfig}

In this section, we introduce the ideals of star configurations in 
$\mathbb{P}^n$. We also provide some  conditions for which these 
ideals are not geometrically vertex decomposable.

Let $\mathbb L = \{L_1,\ldots,L_\ell \}\subseteq R = k[x_0,\dots,x_n]$ 
be a set of  $\ell \geq n$ linear forms, such that any subset of
size less than or equal to $\min\{\ell,n+1\}$ is linearly independent 
in $R_1$.

\begin{df}\label{defn:starconfig}
Fix a positive integer $1 \leq c\leq n$, 
and let $\mathbb L = \{L_1,\ldots,L_\ell \}\subseteq R$ be as above.
We define the ideal
$$I_{\mathbb{X}(\ell,c)} = 
\bigcap_{\sigma = \{j_1,\ldots,j_{c}\} \subseteq [\ell]} \left( L_{j_{1}},\ldots,L_{j_{c}} \right),$$
where $\sigma $ is a subset of $ [\ell]=\{1,\dots,\ell\}$ with $|\sigma| = c$.
The variety $\mathbb{X}(\ell, c)$ defined by $I_{\mathbb{X}(\ell,c)}$ is called
a {\it star configuration} in $\mathbb{P}^n$ of codimension 
$c$ associated to $\mathbb L$.
\end{df}

Note that when $c=n$, then $\mathbb{X}(\ell, n)$ is a zero-dimensional
scheme in $\mathbb{P}^n$.
In  fact, $\mathbb{X}(\ell,n)$ is a set of 
$\binom{\ell}{n}$ points in  $\mathbb{P}^n$ called a {\it star configuration of points} in $\mathbb{P}^n$. 

Ideals of star configurations of arbitrary 
codimension were studied in \cite{GHM13}, 
where their $h$-vector and minimal generators were determined.

\begin{prop}{\cite[Proposition 2.9]{GHM13}}\label{hvec}
  Let $\mathbb{X}(\ell,c) \subseteq \mathbb{P}^n$ be a 
  star configuration of codimension $c$ associated to
  $\mathbb L = \{L_1,\ldots,L_\ell \}$.
  Then
  \begin{enumerate}
      \item the $h$-vector of $\mathbb{X}(\ell,c)$ has exactly $\ell - c + 1$ entries and is equal to: $$\left(1, \binom{c}{c-1},\binom{c+1}{c-1}, \dots, \binom{\ell-1}{c-1}\right).$$
%\item  $deg \mathbb{X}(\ell,c) =\binom{\ell}{c}$.
\item the minimal generators of $I_{\mathbb{X}(\ell,c)}$ are
$$\{L_{j_1}\cdots L_{j_{\ell-c+1}} ~|~ \{j_1,\ldots,j_{\ell-c+1}\} \subseteq [\ell]\}.$$
In particular, all the minimal generators of $I_{\mathbb{X}(\ell,c)}$ have
degree $\ell-c+1$.
  \end{enumerate}
\end{prop}

As part of our goal to prove Theorem \ref{mainthm},
we first consider the case when the number of linear forms is strictly greater than $n+1$.  The codimension $n$ case follows directly from
the next remark.

\begin{rmk}\label{rmk.codim_n_case}
 Since the star configuration of points $\mathbb{X}(\ell,n)$
has $\binom{\ell}{n} > n+1$ points whenever $\ell > n+1$, we obtain as an immediate consequence from Theorem \ref{points} that the ideal
of a star configuration of points with $\ell > n+1$ is not geometrically vertex decomposable.  
\end{rmk}

The next theorem generalizes Remark \ref{rmk.codim_n_case} to 
star configurations of arbitrary codimension with $\ell>n+1$.

\begin{thm}\label{case>n+1}
    For any star configuration $\mathbb{X}(\ell,c) \subseteq 
    \mathbb{P}^n$, if $\ell > n+1$, then $I_{\mathbb{X}(\ell,c)}$ is
    not geometrically vertex decomposable. 
\end{thm}
\begin{proof}    We first show that the condition of being 
almost Hilbertian (see Definition \ref{Hilbertian}) forces $\ell \le n+1$. Therefore, 
when $\ell > n+1$, the ideal $I_{\mathbb{X}(\ell,c)}$ 
cannot be geometrically vertex decomposable by Theorem \ref{hilbertian}.
   
Let $HP(t)$ denote the Hilbert polynomial of 
$R/I_{\mathbb{X}(\ell,c)}$. Since $HP(t)$ has degree $n-c$, 
its $(n-c)$-th difference is constant, i.e.,
$\Delta^{n-c}HP(t)=C$ for some constant $C$.  Here 
$\Delta HP(t) = HP(t) - HP(t-1)$ and $\Delta^iHP(t) =
\Delta^{i-1} HP(t) - \Delta^{i-1}HP(t-1)$ for $i \geq 2$.

Assume that $R/I$ with $I=I_{\mathbb{X}(\ell,c)}$ is almost Hilbertian. Then 
$HP_{R/I}(t) = HF_{R/I}(t)$ for all $t \geq 1$, and therefore the $(n-c)$-th difference of $HF_{R/I}(t)$ satisfies
 $$\Delta^{n-c}HF_{R/I}(t) = \Delta^{n-c}HP_{R/I}(t)=C \qquad \text{for all } t \geq n-c+1.$$
 It follows that the $(n-c+1)$-st difference of the Hilbert function satisfies
 $$\Delta^{n-c+1}HP_{R/I}(t)=0 \qquad  \text{for all } t \geq n-c+2.$$
 By definition, the $(n-c+1)$-st difference of the Hilbert function of $R/I$ is the $h$-vector of $\mathbb{X}(\ell,c)$, and hence this vector can have at most $n-c+2$ nonzero entries.
    
By Proposition \ref{hvec}(1), the $h$-vector of a star configuration $\mathbb{X}(\ell,c)$ has exactly $\ell-c+1$ entries. It follows that, in order for $R/I_{\mathbb{X}(\ell,c)}$ to be almost Hilbertian, we must have
$\ell-c+1 \leq n-c+2$, and therefore $\ell \leq n+1$.
\end{proof}

When $\ell \leq n+1$, the situation is more subtle: the ideals of
star configurations  may or may not be geometrically vertex decomposable depending on $\ell$ and $c$ as well as the particular chosen set of linear 
forms $\mathbb L$.  Two examples that illustrate this 
behaviour were already given in Example \ref{ex.intro}.  
We work out the details for the
third example given in the introduction.
%Particularly interesting is the case $c=\ell-1$, where the GVD property depends on choices of the linear forms. 
%To demonstrate this, we present some examples of star configurations of points in $\mathbb{P}^2$ with $\ell=n+1=3$. 
All computations in the following examples were performed using \texttt{Macaulay2} \cite{M2}.

\begin{ex} \label{exGVD} 
Consider the star configuration of points in $\PP^2$ defined by the 
set of linear forms $\mathbb L=\left\{x_2,x_1+x_2,x_0+x_1+x_2\right\}$. 
The ideal of the star configuration in this case is
 $$I:=I_{\mathbb{X}(3,2)} = \left(x_1\,x_2+x_2^2,\,x_0\,x_2+x_1\,x_2+x_2^2,\,x_0\,x_1+x_1^2+x_0\,x_2+2\,x_1\,x_2+x_2^2\right).$$

Set $y=x_0$.  The lexicographic order $<$ on the monomials 
of $k[x_0,x_1,x_2]$ with $x_0 > x_1 > x_2$ is a  $y$-compatible order.  A Gröbner basis for $I$ with respect to this order is
$$\mathcal{B}=\left\{x_{1}x_{2}+x_{2}^{2}, \ x_{0}x_{2}, 
\ x_{0}x_{1}+x_{1}^{2}-x_{2}^{2}\right\}.$$
Rewriting each basis element as $(x_0)^{\delta_i}q_i+r_i$, 
with $\delta_i=0,1$, we obtain
$$\mathcal{B}=\left\{x_0^0(x_{1}x_{2}+x_{2}^{2}), \ x_{0}^1(x_{2}), \ x_{0}^1(x_{1})+x_{1}^{2}-x_{2}^{2}\right\}.$$
By Definition~\ref{dfGVD}, $C_{y,I}$ is generated by the $q_i$'s, while $N_{y,I}$ is generated only by the $q_i$'s such that $\delta_i=0$. Hence,
    $$C_{y,I}=\left(x_2,x_1, x_{1}x_{2}+x_{2}^{2}\right)=\left(x_1,x_2\right)\ ~\mbox{and}~\ N_{y,I}=\left(x_{1}x_{2}+x_{2}^{2}\right).$$ 
This gives a geometric vertex decomposition of $I$ with respect to $y$. 
Since $C_{y,I}$ is generated by indeterminates, it is 
geometrically vertex decomposable. For $N_{y,I}$, the generators already 
form a Gröbner basis for the lexicographical order on $k[x_1,x_2]$ that is compatible with $y'=x_1$. Then
 $$C_{y',N_{y,I}}=(x_{2}) ~\mbox{and}~ N_{y',N_{y,I}}=(0)$$ 
give us a geometric vertex decomposition of  $N_{y,I}$ with respect to $y'$.
Since the zero ideal and $(x_2)$ are geometrically
vertex decomposable, we conclude that $N_{y,I}$ and, as a consequence, $I$ are geometrically vertex decomposable.
\end{ex}

\begin{ex}
Consider the ideal of the star configuration of points in $\PP^2$ 
constructed
from $\mathbb L=\left\{x_2,x_0+x_1, x_0+2x_1+3x_2\right\}$.  The defining
ideal is
    $$I_{\mathbb{X}(3,2)}=\left(x_{1}^{2}-2\,x_{0}x_{2}-3\,x_{2}^{2},\,x_{0}x_{1}+
      3\,x_{0}x_{2}+x_{1}x_{2}+3\,x_{2}^{2},\,x_{0}^{2}-2\,x_{0}x_{2}-3\,x_{2
      }^{2}\right).$$
      In this case, since for any $x_i$ we have an element in $I$ of degree two that is not squarefree in $x_i$, we conclude by Lemma \ref{mindeg} that we cannot have a Gröbner basis that is squarefree in any variable $x_i$. Hence, by Lemma~\ref{lemma:squarefree}, we cannot obtain a geometric vertex decomposition with respect to any of the variables, so the ideal of the star configuration is not
      geometrically vertex decomposable.
\end{ex}

\begin{ex} 
As we illustrate in this example, sometimes
one may need to reiterate the procedure a few times before
determining whether or not an ideal is geometrically vertex
decomposable.  In particular, consider the star configuration $\mathbb{X}(3,2)$ in $\PP^2$ defined by 
    $$\mathbb L=\left\{x_0+4x_1,x_0+x_1,x_0+2x_1+3x_2\right\}.$$ 
     In this case, although the associated ideal $I$ admits a geometric vertex decomposition, it is not geometrically vertex decomposable.
    %$$I=\left(2x_1^2+3x_0x_2+9x_1x_2,x_0x_1-3x_0x_2-6x_1x_2,x_0^2+9x_0x_2+12x_1x_2\right).$$
    In particular, setting $y=x_2$, we use the lexicographical order based on $x_2>x_1>x_0$. 
This ordering is a $y$-compatible monomial order, and we obtain 
the following Gröbner basis for $I$:
    $$\mathcal{B}=\left\{x_2(3x_0)+2x_0x_1+x_0^2, \ x_2^0(4x_{1}^{2}+5x_{0}x_{1}+x_{0}^{2}), \ x_2(6x_{1})-3x_{0}x_{1}-x_{0}^{2}\right\}.$$
Then we have
    $$C_{y,I}=\left(3x_0,6x_1, x_0^2+5x_0x_1+4x_1^2\right)=\left(x_0,x_1\right)~\mbox{and}~ N_{y,I}=\left(x_0^2+5x_0x_1+4x_1^2\right).$$    

This gives a geometric vertex decomposition of $I$. While $C_{y,I}$ is
geometrically vertex decomposable, the ideal $N_{y,I}$ is not geometrically vertex decomposable by Lemma \ref{lemma:squarefree}, since we cannot find squarefree Gröbner bases in any variable. 
Therefore, the ideal $I$ is not geometrically
vertex decomposable.

\end{ex}

%%%%%%%%%%%%%%%%%%%%%%%%%%%%%%%%%%%%%%%%%%%%%

\section{Classification of GVD Star Configurations}\label{sec.classification}

In this section we prove Theorem~\ref{mainthm}, a classification
of the ideals of star configurations that are geometrically
vertex decomposable.
We first establish some lemmas and properties that will be used in the proof.

\begin{lemma}\label{induction}
    Let $\mathbb L = \{L_1,\ldots,L_\ell \}\subseteq R = k[x_0,\dots,x_n]$ 
    be a set of  linear forms such that any subset of
    size less than or equal to $\min\{\ell,n+1\}$ is linearly independent.
    Assume $c<\ell$. Then for any $i=1,\dots,\ell$, the ideal of the star configuration of codimension $c$ associated to $\mathbb L$ can be written as
    $$I_{\mathbb{X}(\ell,c)} = \left(I_{\mathbb{X}(\ell-1,c)}\cap(L_i)\right)+ I_{\mathbb{X}(\ell-1,c-1)},$$
    where $I_{\mathbb{X}(\ell-1,c)}$ and $I_{\mathbb{X}(\ell-1,c-1)}$  are the ideals of the star configuration associated to $\mathbb L \setminus \{L_i\}$ of codimensions $c$ and $c-1$, respectively.
\end{lemma}
\begin{proof}
    Without loss of generality, we can assume $L_i=L_\ell$. We have
    \begin{align*}
I_{\mathbb{X}(\ell,c)}  & = \bigcap_{\{j_1,\dots,j_c\} \subseteq [\ell]}(L_{j_1}, \ldots, L_{j_c}) \\
& = \left[\bigcap_{\{j_1,\dots,j_c\} \subseteq [\ell-1]}(L_{j_1}, \ldots, L_{j_c})\right] \cap \left[\bigcap_{\{j_1,\dots,j_{c-1}\} \subseteq [\ell-1]}(L_{j_1}, \ldots, L_{j_{c-1}}, L_\ell)\right] \\
& = \left[\bigcap_{\{j_1,\dots,j_c\} \subseteq [\ell-1]}(L_{j_1}, \ldots, L_{j_c})\right] \cap \left[( L_\ell)+\bigcap_{\{j_1,\dots,j_{c-1}\} \subseteq [\ell-1]}(L_{j_1}, \ldots, L_{j_{c-1}})\right]
\\
& = \left[( L_\ell) \cap \bigcap_{\{j_1,\dots,j_c\} \subseteq [\ell-1]}(L_{j_1}, \ldots, L_{j_c})\right] + \left[\bigcap_{\{j_1,\dots,j_{c-1}\} \subseteq [\ell-1]}(L_{j_1}, \ldots, L_{j_{c-1}})\right],
\end{align*}
where we used the fact that if $K \subseteq I$, then
$I\cap (J+K)  = (I \cap J) + K$ in the fourth equality.

Finally we note that $\bigcap_{\{j_1,\dots,j_c\} \subseteq [\ell-1]}(L_{j_1}, \ldots, L_{j_c})$ and $\bigcap_{\{j_1,\dots,j_{c-1}\} \subseteq [\ell-1]}(L_{j_1}, \ldots, L_{j_{c-1}})$ are the ideal of the star configuration of codimension $c$ and $c-1$, respectively,  associated  
to the set of linear forms $\{L_1,\ldots, L_{\ell-1}\}.$
\end{proof}

%From the work of \cite{GHM13} we know that the ideal $I_{\mathbb{X}(\ell,c)}$ is generated in degree $\ell-c+1$ by products of the linear form in $\mathbb L$, as a consequence, we have the following lemma.

\begin{lemma}\label{var2}
Let $\mathbb L = \{L_1,\ldots,L_\ell \}\subseteq R = k[x_0,\dots,x_n]$ 
    be a set of  $\ell \geq n$ linear forms such that any subset of
    size less than or equal to $\min\{\ell,n+1\}$ is linearly independent.
    Suppose that $c+1 \leq \ell$, and for each 
    $x_i \in \{x_0,\ldots,x_n\}$, there are at least two linear forms of $\mathbb L = \{L_1,\ldots,L_\ell \}$ such that the coefficient of $x_i$ is not zero. 
    Then $I_{\mathbb{X}(\ell,c)}$ is not GVD.
 \end{lemma}
\begin{proof}
    By Lemma \ref{lemma:squarefree}, it is enough to show that for any $y = x_i$, there does not exist a Gröbner basis that is squarefree in $y$. 

    Without loss of generality, we can assume that $y$ appears with a nonzero coefficient in $L_1$ and $L_2$. 
   
By Proposition \ref{hvec} (2),  $L_1\cdot L_2 \cdots L_{\ell-c+1}$ is a generator of $I_{\mathbb{X}(\ell,c)}$ of degree $\ell-c+1 \geq  2$ and it is not squarefree in $x_i$ since $x_i$ appears in both $L_1$ and $L_2$. 
Since  all generators have degree $\ell-c+1$, we can apply Lemma \ref{mindeg} to conclude that there does not exist a Gröbner basis squarefree in $y$. 
\end{proof}

We now come to the main result of this section, namely 
the classification of geometrically vertex decomposable ideals of star configurations.

\begin{thm}\label{maintheorem4} Let $\mathbb X(\ell,c)$ be a star 
configuration of codimension $1 \leq c \leq n$ in $\mathbb{P}^n$ 
defined by the general linear forms 
$\mathbb L = \{L_1, \ldots, L_{\ell}\}$ with $\ell\geq c$.  
\begin{enumerate}
    \item If $\ell = c$, then $I_{\mathbb{X}(\ell,c)}$ is always
    a geometrically vertex decomposable ideal.
    \item If $\ell =c+1$, then $I_{\mathbb{X}(\ell,c)}$ is a geometrically
    vertex decomposable ideal if and only if (up to reordering the variables and the linear forms in $\mathbb L$) the coefficients of the linear forms in $\mathbb L$ form a matrix that has a $(c+1)\times(c+1)$ triangular submatrix
$$
\left[ \begin{array}{ccccccc}
0 & 0 & \cdots & a_{1,c}&*& \cdots &* \\
\vdots & \vdots & \vdots & \vdots&\vdots& \vdots &\vdots  \\
0 & a_{c,1} &  \cdots & a_{c,c}  & *& \cdots &* \\
a_{c+1,0}& a_{c+1,1}& \cdots   & a_{c+1,c}&*& \cdots &* 
\end{array}
\right].$$
\item If $c+1<\ell \leq n+1$, then $I_{\mathbb{X}(\ell,c)}$ is a 
    geometrically vertex decomposable ideal if and only if 
    $I_{\mathbb{X}(\ell,\ell-1)}$ is a geometrically vertex
    decomposable ideal.   
    %the ideal of the star configuration of codimension $\ell-1$ defined by $\mathbb L$ is geometrically vertex decomposable.
   \item If $\ell > n+1$, then $I_{\mathbb{X}(\ell,c)}$ is never a
    geometrically vertex decomposable ideal. 
\end{enumerate}
\end{thm}

\begin{proof}
Part (1) follows from Lemma~\ref{CI}, since $I_{\mathbb{X}(\ell,c)}$ is a complete intersection of linear forms when $\ell=c$.
Part (4) is proved in Theorem~\ref{case>n+1}.
We treat (2) and (3) together. We need to show that for $c+1\leq\ell \leq n+1$, the ideal $ I_{\mathbb{X}(\ell,c)}$ is geometrically vertex decomposable if and only if (up to reordering the variables and the linear forms in $\mathbb L$) the coefficients of the linear forms in $\mathbb L$ form a matrix $A$ that contains a $\ell \times \ell$ triangular submatrix
$$
\left[ \begin{array}{ccccccc}
0 & 0 & \cdots & a_{1,\ell-1}&*& \cdots &* \\
\vdots & \vdots & \vdots & \vdots&\vdots& \vdots &\vdots  \\
0 & a_{\ell-1,1} &  \cdots & a_{\ell-1,\ell-1}  & *& \cdots &* \\
a_{\ell,0}& a_{\ell,1}& \cdots   & a_{\ell,\ell-1}&*& \cdots &* 
\end{array}
\right].$$

We will proceed on double induction: 
\begin{itemize}
    \item on $\ell$, with base case $\ell=2$;
    \item and on $\ell-c$, with base case $\ell=c+1$ (that corresponds to part (2) of the statement).
\end{itemize}
Let us first consider the case $\ell=2$. 
Since $1\leq c\leq \ell-1$, we must have $c+1=\ell=2$. 
Assume $L_1=a_0x_0+\cdots +a_nx_n$ and  $L_2=b_0x_0+\cdots + b_nx_n$. Then
$$I_{\mathbb{X}(\ell,c)} = \left(\sum_{i=0}^{n} a_ix_i\right)\cap\left(\sum_{j=0}^{n} b_jx_j\right)= \left(\sum_{i=0}^{n}\sum_{j=0}^{n} a_ib_jx_ix_j\right).$$
It is enough to show that $I_{\mathbb{X}(\ell,c)}$ is 
geometrically vertex decomposable if and only if at least one of the coefficients $\{a_0,\ldots,a_n,b_0,
\ldots,b_n\}$ is zero.
It is clear that if all $a_i,b_j$ are nonzero, the generator of $I_{\mathbb{X}(\ell,c)}$ is not squarefree in any of the variables, and thus by Lemma
\ref{mindeg} and Lemma \ref{lemma:squarefree}, the ideal
is not geometrically vertex decomposable. 
Conversely, assume $a_0=0$ and set $y=x_0$. Then 
$$I_{\mathbb{X}(\ell,c)} = \left(\sum_{i=1}^{n}a_ib_0x_iy+\sum_{i=1}^{n}\sum_{j=1}^{n} a_ib_jx_ix_j\right)$$
and, for any $y$-compatible order, by Lemma \ref{lemma: indep.} we have  
$$C_{y,I}=({\rm in}_y(I) :y^{\infty})= \left(\left(\sum_{i=1}^{n}a_ib_0x_iy \right):y^{\infty}\right)=\left(\sum_{i=1}^{n}a_ib_0x_i\right), ~~\mbox{and}~~ N_{y,I}=(0).$$ 
This is a geometric vertex decomposition since 
$${\rm in}_y(I)=\left(\sum_{i=1}^{n}a_ib_0x_i\right) \cap (y)= C_{y,I}\cap (N_{y,I} + (y)).$$
The ideal $C_{y,I}$ is then geometrically vertex
decomposable by \cite[Lemma 2.6]{CDSRVT}.
We conclude that 
our base case $I_{\mathbb{X}(2,1)}$ is geometrically
vertex decomposable.

Let us now assume $\ell > 2$. 
First, by Lemma \ref{var2}, if every variable appears in at least 2 linear forms, then $I_{\mathbb{X}(\ell,c)}$ is not geometrically vertex decomposable. 
In this case, the matrix of coefficients cannot contain a triangular submatrix.

Therefore, we can assume that the variable $y=x_0$ only appears in the linear form $L_\ell$, i.e., $y$ has a nonzero
coefficient only in $L_\ell$.
By Lemma \ref{induction} we have that $$I=I_{\mathbb{X}(\ell,c)} = \left(I_{\mathbb{X}(\ell-1,c)}\cap(L_\ell)\right)+ I_{\mathbb{X}(\ell-1,c-1)}.$$ Since the variable $y=x_0$ appears only in $L_\ell$, we can consider $I_{\mathbb{X}(\ell-1,c)}$ and $I_{\mathbb{X}(\ell-1,c-1)}$ as ideals of star configurations in $\mathbb P^{n-1}$ with respect to the set of linear forms $\{L_1,\ldots, L_{\ell-1}\}.$

For a $y$-compatible monomial order, we have that
$${\rm in}_y(I)= \left(I_{\mathbb{X}(\ell-1,c)}\cap(y)\right)+ I_{\mathbb{X}(\ell-1,c-1)}$$
and so by Lemma \ref{lemma: indep.} we have
$$C_{y,I}=({\rm in}_y(I) :y^{\infty})=I_{\mathbb{X}(\ell-1,c-1)} + I_{\mathbb{X}(\ell-1,c)}=I_{\mathbb{X}(\ell-1,c)}.$$
Moreover, 
$$N_{y,I}+(y)= {\rm in}_y(I) +(y)= (y)+ I_{\mathbb{X}(\ell-1,c-1)},$$
and since $y$ does not appear in the generators of $N_{y,I}$ we conclude that
$$N_{y,I}=I_{\mathbb{X}(\ell-1,c-1)}.$$
This gives a geometric vertex decomposition of $I_{\mathbb{X}(\ell,c)}$; in fact
$${\rm in}_y(I) = [(y)\cap C_{y,I}] + N_{y,I} =
C_{y,I} \cap [(y)+N_{y,I}] = I_{\mathbb{X}(\ell-1,c)}
\cap [(y) + I_{\mathbb{X}(\ell-1,c-1)}].$$

We now proceed by induction on $\ell-c$.  
First, consider the base case $\ell=c+1$. We use induction again, this time on  $\ell$. We have already proved the statement for $\ell=2$, so we can assume $\ell>2$.  Observe that
$N_{y,I}=I_{\mathbb{X}(\ell-1,c-1)}$ is a star configuration of codimension $c-1$ defined by a set of $\ell=c+1$ lines. By case (1) of this theorem, it is always geometrically vertex decomposable.
For $C_{y,I}=I_{\mathbb{X}(\ell-1,c)},$ since $(c-1)+1=\ell-1$ we can apply the inductive hypothesis to conclude that $C_{y,I}$ is geometrically vertex
decomposable if and only if the coefficient matrix $A'$ has a $(\ell-1)\times (\ell-1)$ triangular submatrix.
Note that $A'$ is obtained from $A$ by deleting the last row and the first column. Hence $I_{\mathbb{X}(\ell,c)}$ is geometrically vertex decomposable if and only if $A$ contains a triangular $\ell\times \ell$ submatrix.

This concludes the base case. For any $c+1<\ell \leq n+1$, we can use the inductive hypothesis (on $\ell-c$ and on $\ell$)  to conclude that $N_{y,I}=I_{\mathbb{X}(\ell-1,c-1)}$ and $C_{y,I}=I_{\mathbb{X}(\ell-1,c)}$ are 
geometrically vertex decomposable if and only if their matrices of coefficient $A'$ have a $(\ell-1)\times (\ell-1)$  triangular submatrix. Since $A'$ is obtained from $A$ by deleting the last row and the first column, and we already know that the first column contains at most one nonzero entry, we conclude that  $I_{\mathbb{X}(\ell,c)}$ is
geometrically vertex decomposable if and only if $A$ has a triangular $\ell\times \ell$ submatrix.
\end{proof}

%%%%%%%%%%%%%%%%%%%%%%%%%%%%%%%%%%%%%%%%%%%%%%

\section{Star configurations and Knutson ideals}\label{sec.knutson}

In his work on Frobenius splittings, Knutson
\cite{K2009} introduced a family of ideals
that can be constructed from a polynomial whose
leading term is a squarefree monomial in some 
monomial order.  These ideals were later 
called Knutson ideals by Conca-Varbarro \cite{CV2020}.
As recently shown in \cite{DGKRS}, there appears 
to be a relationship
between Frobenius splittings, Knutson ideals, and 
geometrically vertex decomposable ideals.
Ideals of star configurations give another example
of this connection.  
In fact, for ideals of star configurations, being a Knutson
ideal is the exact same as being a geometrically 
vertex decomposable ideal.

We recall the definition of Knutson ideals
as found in \cite{CV2020}.

\begin{df}\label{kn}
    Let $f\in R = k[x_0,\dots, x_n]$ be a polynomial such that its leading term ${\rm in}_<(f)$ is squarefree with respect to a fixed monomial order $<$. Define $C_f$ to be the smallest set of ideals satisfying the following conditions:
    \begin{enumerate}[label=\Roman*.]
        \item $(f)\in C_f$;
        \item if $I\in C_f$, then $I:J\in C_f$ for every ideal $J\subseteq R$;
        \item if $I, J \in C_f$, then $I+J \in C_f;$ and
        \item if $I, J \in C_f$, then $I\cap J \in C_f.$ 
    \end{enumerate}
    An ideal $I\in C_f$ is called a {\it Knutson ideal associated to $f$.}
\end{df}

Knutson ideals satisfy the property that their initial ideals
are squarefree, as given in the next result.

\begin{thm}\cite[Main Theorem 1.]{S2021}\label{initial ideal}
Let $f \in R = k[x_0,\ldots,x_n]$ be a polynomial
such that ${\rm in}_<(f)$ is squarefree with respect
to a monomial order $<$.  If $I \in C_f$, then the initial
ideal ${\rm in}_<(I)$ is a squarefree monomial ideal.
\end{thm}

We can apply the previous result to show that
ideals of star configurations cannot be Knutson ideals
if $\ell$ is ``too large''.

\begin{thm}\label{kn l>n+1}
Let $\mathbb X(\ell,c)$ be a star 
configuration of codimension $1 \leq c \leq n$ in $\mathbb{P}^n$
defined by the general linear forms 
$\mathbb L = \{L_1, \ldots, L_{\ell}\}$ with $\ell\geq c$.
If $\ell>n+1$, then $I_{\mathbb{X}(\ell,c)}$ is not a Knutson ideal.
\end{thm}

\begin{proof}
Fix a monomial order $<$.  Since $\deg(L_i)=1$ for all
$i=1,\ldots,\ell$, we have ${\rm in}_<(L_i) =x_{j_i}$ 
for some variable $x_{j_i} \in \{x_0,\ldots,x_n\}$.  Because
$\ell >n+1$, there exist as least two $L_i,L_j \in \mathbb{L}$
with $L_i \neq L_j$, but ${\rm in}_<(L_i) = {\rm in}_<(L_j)$.

By Proposition \ref{hvec}, the ideal $I = I_{\mathbb{X}(\ell,c)}$
is generated by all elements of the form
$$L_{i_1}L_{i_2}\cdots L_{i_{\ell-c+1}} ~~\mbox{with 
$\{i_1,\ldots,i_{\ell-c+1}\} \subseteq [l]$.} $$
Since $c \leq n$, we have $\ell-c+1 \geq \ell -n+1 > \ell -n-1 >0$,
and hence $\ell-c+1 \geq 2$.  Consequently, there is 
a generator of $I_{\mathbb{X}(\ell,c)}$ of the form
$$G = L_iL_jL_{i_{3}}\cdots L_{i_{\ell-c+1}} ~~\mbox{with 
$\{i_3,\ldots,i_{\ell-c+1}\} \subseteq  [\ell] \setminus \{i,j\}$.}
$$
Thus ${\rm in}_<(G) = {\rm in}_<(L_i){\rm in}_<(L_j) 
{\rm in}_<(L_{i_3})\cdots {\rm in}_<(L_{i_{\ell-c+1}})$ 
is not a squarefree monomial
since ${\rm in}_<(L_i) = {\rm in}_<(L_j)$. 
Furthermore ${\rm in}_<(G)$ must be a minimal generator
of ${\rm in}_<(I)$.  To see why, suppose there was 
a squarefree monomial generator ${\rm in}_<(I)$ that
divides ${\rm in}_<(G)$.  Its degree must be strictly less
than $\ell-c+1$.  But since $I$ is generated
by homogeneous elements of  degree $\ell-c+1$
by Proposition \ref{hvec}, the initial ideal
cannot have any monomials of degree strictly less
than $\ell-c+1$. 

Theorem \ref{initial ideal} implies that if an ideal
is Knutson, there is at least one monomial order 
such that the initial ideal is squarefree.  However, we have 
shown that for any monomial order $<$, the initial ideal of
$I_{\mathbb{X}(\ell,c)}$ is not squarefree.
So $I_{\mathbb{X}(\ell,c)}$ 
cannot be a Knutson ideal.
\end{proof}

We now come to the main theorem of this section.

\begin{thm}\label{mainthm2}
    The ideal $I_{\mathbb{X}(\ell,c)}$ of a star configuration of codimension $c$ in $\mathbb{P}^n$     
    %in $I_{\mathbb{X}%(\ell,c)}\subseteq\mathbb{P}^n$
    is geometrically vertex decomposable if and only if it is a Knutson ideal.
\end{thm}
\begin{proof}
    Let $\mathbb X(\ell,c)$ be a star configuration of codimension $c$ in $\mathbb{P}^n$ defined by the general linear forms $\mathbb L = \{L_1, \ldots, L_{\ell}\}$ with $\ell\geq c$.  Let $I=I_{\mathbb{X}(\ell,c)}$ be the associated ideal.
    By Theorem \ref{maintheorem4} we want to show that $I$ 
    is a Knutson ideal if and only if 
    \begin{enumerate}
        \item $c=\ell$, or
        \item $c+1\leq\ell\leq n+1$ and the coefficients of the linear forms in $\mathbb L$ (up to reordering the variables and the linear forms) form a matrix $A$ that contains an $\ell \times \ell$ triangular submatrix
$$
\left[ \begin{array}{ccccccc}
0 & 0 & \cdots & a_{1,\ell-1}&*& \cdots &* \\
\vdots & \vdots & \vdots & \vdots&\vdots& \vdots &\vdots  \\
0 & a_{\ell-1,1} &  \cdots & a_{\ell-1,\ell-1}  & *& \cdots &* \\
a_{\ell,0}& a_{\ell,1}& \cdots   & a_{\ell,\ell-1}&*& \cdots &* 
\end{array}
\right].$$
    \end{enumerate}
Note that we already know by %Lemma
Theorem \ref{kn l>n+1} that when $\ell>n+1$, the ideal $I$ is not a Knutson ideal, so we can restrict to the cases $c\leq \ell\leq n+1$.

First, let us consider the case $c=\ell$. In this case, proceeding as in Lemma~\ref{CI}, we can find new linear forms $L'_1,\dots,L'_c$ such that $I=(L'_1,\dots,L'_c)$ and the coefficients are in a triangular matrix. Then $f=L'_1L'_2\cdots L'_c$ has 
squarefree leading term ${\rm in}_<(f)$ when an appropriate monomial
order is chosen.
Now, by Definition \ref{kn} I. and II., 
for any $i=1,\dots, c$ we have
$$(L'_i)=(f):(L'_1\cdots L'_{i-1}L'_{i+1}\cdots L'_c)\in C_f.$$
Using part III. of Definition \ref{kn} we conclude that
$$I=(L'_1,\dots,L'_c)=(L'_1)+\dots+(L'_c)\in C_f.$$  
Therefore $I$ is a Knutson ideal associated to $f$.

Let us now assume $c+1\leq\ell\leq n+1$. If the coefficients of the linear forms $L_1,\dots, L_\ell$ form a matrix that contains an $\ell \times \ell$ triangular submatrix, then the proof is similar to what was done in the case $\ell=c$. In fact, for an appropriate choice of monomial order $<$, 
$f=L_1L_2\cdots L_c$ has squarefree leading term and 
the following ideals are all Knutson ideals associated
to $f$:
\begin{align*}
   (L_i)=(f):(L_1\cdots L_{i-1}L_{i+1}\cdots L_\ell)\in C_f \ &\text{for any $i\in [\ell]$, by  Definition \ref{kn} I. and II.} \\
   (L_{j_1}, \dots, L_{j_c})=  (L_{j_1})+ \dots +(L_{j_c}) \in C_f \ &\text{for any $\{j_1,\dots,j_c\}\subseteq [\ell]$ by Definition \ref{kn} III.}\\
   I=\bigcap_{\{j_1,\dots,j_c\} \subseteq [\ell]}(L_{j_1}, \ldots, L_{j_c})\in C_f  \ &\text{by Definition \ref{kn} IV.}
\end{align*}
Therefore, we conclude that in this case $I$ is a Knutson ideal. 

Conversely, assume that the matrix of coefficients does not contain an $\ell \times \ell$ triangular submatrix. This implies that for any monomial order $<$, there exist two linear forms $L_t$ and $L_j$ such that ${\rm in}_<(L_t)={\rm in}_<(L_j)=x_i$. Then proceeding in a similar way as  in Theorem \ref{kn l>n+1} we obtain an element of $I$ of degree $\ell-c+1$ whose initial term is not squarefree in $x_i$. By the minimality of the degree $\ell-c+1$, we conclude that 
${\rm in}_<(I)$ is not squarefree.  Hence $I$ is not a  Knutson ideal by Theorem \ref{initial ideal}.
\end{proof}

\begin{rmk}\label{gvd+knutson}
 The connection between geometrically vertex decomposable
 ideals and Knutson ideals is further studied in
 \cite{DGKRS}.  In general, there are
 geometrically vertex decomposable ideals that are 
 not Knutson; for one example, see \cite[Example 3.3]{DGKRS}.
 Similarly, there are Knutson ideals that are 
 geometrically vertex decomposable.  One such
 example is \cite[Example 3.14]{CV2020}, which is
 an example of a Knutson ideal that is not Cohen-Macaulay. 
 Since all geometrically vertex decomposable ideals are 
 Cohen-Macaulay, this is an example of a Knutson ideal that 
 cannot be geometrically vertex decomposable.
\end{rmk}

%%%%%%%%%%%%%%%%%%%%%%%%%%%%%%

\section{Future directions}\label{sec.future}

The star configurations studied in this paper are 
sometimes called {\it linear} star configurations
since the variety is constructed from a set $\mathbb{L} =\{L_1,\ldots,
L_\ell\}$ of linear forms.
A generalization of star configurations, 
which relaxes the linear  condition,
was introduced by Geramita, Harbourne,
Migliore, and Nagel
\cite{GHMN2017}.  We recall the definition here:

\begin{df}\label{defn:genstarconfig}
Let $\mathbb{F} = \{F_1,\ldots,F_\ell\}$
be homogeneous forms of $R$ with 
$\deg(F_i) = d_i$ for $i=1,\ldots,\ell$.  
Fix a positive integer $1 \leq c\leq n$,
and assume that for every subset of $c+1$ elements
of $\mathbb{F}$, the ideal generated by
those elements has codimension $c+1$.  
We define the ideal
$$I_{\mathbb{Z}(\ell,c)} = 
\bigcap_{\sigma = \{j_1,\ldots,j_{c}\} \subseteq [\ell]} \left( F_{j_{1}},\ldots,F_{j_{c}} \right),$$
where $\sigma $ is a subset of $ [\ell]=\{1,\dots,\ell\}$.
The variety $\mathbb{Z}(\ell, c)$ defined by $I_{\mathbb{Z}(\ell,c)}$ is called
a {\it general star configuration} in $\mathbb{P}^n$ of codimension 
$c$ associated to $\mathbb F$.
\end{df}

We recover Definition \ref{defn:starconfig}
when all the $F_i$'s in the previous definition 
are linear forms.
In light of the goal of this paper, it is natural
to ask:

\begin{question}\label{openquestion}
    Can we classify the ideals of general star configurations
    that are geometrically vertex decomposable and/or
    Knutson?
\end{question}

Based upon our earlier results,
the answer to this question  could be quite subtle since 
not only will the coefficients be relevant, but 
we also suspect that the degrees of the $F_i$'s will play a role.  
Fleshing out these details is the basis for a  future project.

The following two examples show that
there exist general star configurations 
whose ideals are geometrically vertex decomposable and Knutson.

\begin{ex}

Let $R=k[x_0,\dots,x_5]$, and $\mathbb{Z}(3,2) \subseteq \mathbb{P}^5$ be a 
  general star configuration of codimension $2$ associated to
  $\mathbb F = \{F_1,F_2,F_3\}$,
  \noindent where $F_1=x_0x_1, F_2=x_2x_3$ and $F_3=x_4x_5$ are squarefree monomials.  The general star configuration ideal is
\[
I_{\mathbb{Z}(3,2)} = \bigcap_{1 \le i < j \le 3} (F_i, F_j)
= (F_1,F_2)\cap(F_1,F_3)\cap(F_2,F_3).
\]
The ideal $I=I_{\mathbb{Z}(3,2) }$ is a monomial
ideal minimally generated by the following monomials
\[
I = (
x_0x_1x_2x_3,\;
x_0x_1x_4x_5,\;
x_2x_3x_4x_5
).
\]
Using {\it Macaulay2}, we can verify that this ideal
is geometrically vertex decomposable.  An alternative
way to verify this fact is to show that the simplicial
complex associated to $I$ via the Stanley-Reisner correspondence is a vertex decomposable simplicial complex. 
One can then use
\cite[Proposition 2.9]{KR21} to show that $I$ is a 
geometrically vertex decomposable ideal.

This ideal is also Knutson since every squarefree
monomial ideal is Knutson.
\end{ex}

We now give another non-monomial example of a
generalized star configuration that is geometrically
vertex decomposable.

\begin{ex}
Let $R = k[x_0, x_1, \ldots, x_{11}]$ be the coordinate ring of $\mathbb{P}^{11}$, 
and consider the general star configuration $\mathbb{Z}(3,2)$ associated to
$$\mathbb{F} = \{F_1, F_2, F_3\} = 
\{(x_0+x_1)(x_2+x_3), \, (x_4+x_5)(x_6+x_7), \, (x_8+x_9)(x_{10}+x_{11})\}.$$
Note that $F_1, F_2, F_3$ are non-linear, 
non-monomial forms involving pairwise disjoint sets of variables.
By Proposition~\ref{gens}, the minimal generators of $I_{\mathbb{Z}(3,2)}$ are 
the products of $\ell - c + 1 = 2$ forms of type:
\begin{align*}
F_1 F_2 &= (x_0+x_1)(x_2+x_3)(x_4+x_5)(x_6+x_7), \\
F_1 F_3 &= (x_0+x_1)(x_2+x_3)(x_8+x_9)(x_{10}+x_{11}), \\
F_2 F_3 &= (x_4+x_5)(x_6+x_7)(x_8+x_9)(x_{10}+x_{11}).
\end{align*}
Each $F_iF_j$ is a polynomial whose terms are all
squarefree.
Furthermore, using the lexicographical monomial order
with $x_0 > x_1 > \cdots > x_{11}$, these 
generators form a Gr\"obner basis of the
ideal $I$.

We now verify that $I = I_{\mathbb{Z}(3,2)}$ is geometrically vertex decomposable. 
If we set $y = x_0$, then the above lexicographical order $<$ 
is $y$-compatible. 
Rewriting each generator as $x_0^{\delta_i} q_i + r_i$, we obtain:
\begin{align*}
F_1 F_2 &= x_0 \cdot (x_2+x_3)(x_4+x_5)(x_6+x_7) 
          + x_1(x_2+x_3)(x_4+x_5)(x_6+x_7), \\
F_1 F_3 &= x_0 \cdot (x_2+x_3)(x_8+x_9)(x_{10}+x_{11}) 
          + x_1(x_2+x_3)(x_8+x_9)(x_{10}+x_{11}), \\
F_2 F_3 &= x_0^0 \cdot (x_4+x_5)(x_6+x_7)(x_8+x_9)(x_{10}+x_{11}).
\end{align*}
By Definition~\ref{dfGVD}, this gives:
\begin{align*}
C_{y,I} &= \big((x_2+x_3)(x_4+x_5)(x_6+x_7), \; 
             (x_2+x_3)(x_8+x_9)(x_{10}+x_{11}), \\
         &\qquad (x_4+x_5)(x_6+x_7)(x_8+x_9)(x_{10}+x_{11})\big), \\
N_{y,I} &= \big((x_4+x_5)(x_6+x_7)(x_8+x_9)(x_{10}+x_{11})\big),
\end{align*}
and we have a geometric vertex decomposition
$$\mathrm{in}_y(I) = C_{y,I} \cap (N_{y,I} + (y)).$$

The ideal $N_{y,I}$ is principal and generated by a squarefree polynomial, 
hence geometrically vertex decomposable.
We now show that $C_{y,I}$ is geometrically vertex decomposable.
The generators of $C_{y,I}$ form a Gr\"obner basis
with respect to the lexicographical order
$x_2 > x_3 > \cdots > x_{11}$.
Set $y' = x_2$ and rewrite the generators of $C_{y,I}$ as $x_2^{\delta_i} q_i + r_i$:
\begin{align*}
(x_2+x_3)(x_4+x_5)(x_6+x_7) &= x_2 \cdot (x_4+x_5)(x_6+x_7) 
                                + x_3(x_4+x_5)(x_6+x_7), \\
(x_2+x_3)(x_8+x_9)(x_{10}+x_{11}) &= x_2 \cdot (x_8+x_9)(x_{10}+x_{11}) 
                                + x_3(x_8+x_9)(x_{10}+x_{11}), \\
(x_4+x_5)(x_6+x_7)(x_8+x_9)(x_{10}+x_{11}) &= x_2^0 \cdot 
                                (x_4+x_5)(x_6+x_7)(x_8+x_9)(x_{10}+x_{11}).
\end{align*}
This gives:
\begin{align*}
C_{y',C_{y,I}} &= \big((x_4+x_5)(x_6+x_7), \; (x_8+x_9)(x_{10}+x_{11})\big), \\
N_{y',C_{y,I}} &= \big((x_4+x_5)(x_6+x_7)(x_8+x_9)(x_{10}+x_{11})\big).
\end{align*}
Again, the ideal $N_{y',C_{y,I}}$ is principal and generated by a squarefree polynomial, and hence geometrically 
vertex decomposable.   Moreover, one can verify that
$$\mathrm{in}_{y'}(I) = C_{y',{C_{y,I}}} \cap (N_{y',C_{y,I}} + (y)).$$

Since the generators of $C_{y',C_{y,I}}$ are 
in disjoint sets of variables, 

we have
\footnotesize
$$k[x_3,\ldots,x_{11}]/C_{y',C_{y,I}} \cong
k[x_4,\ldots,x_7]/((x_4+x_5)(x_6+x_7))\otimes k[x_8,\ldots,x_{11}]/((x_8+x_9)(x_{10}+x_{11})).$$
\normalsize
Because the ideal $((x_4+x_5)(x_6+x_7))$,
respectively $((x_8+x_9)(x_{10}+x_{11}))$, is
geometrically vertex decomposable in $k[x_4,\ldots,x_7]$,
respectively $k[x_8,\ldots,x_{11}]$, the ideal
$C_{y',C_{y,I}}$ is now geometrically vertex decomposable
by \cite[Theorem 2.9]{CDSRVT}.  Consequently, the ideal
$I$ is geometrically vertex decomposable, as desired.
\end{ex}

We conclude with some preliminary steps towards answering Question \ref{openquestion}.  In particular,
we show that some of our earlier arguments can be easily
adapted to general star configurations to give
examples of ideals that fail to be geometrically vertex decomposable
or Knutson.  We require the following generalization
of Proposition \ref{hvec}:

\begin{prop}{\cite[Proposition 2.3]{GHMN2017}}\label{gens}
  Let $\mathbb{Z}(\ell,c) \subseteq \mathbb{P}^n$ be a 
  general star configuration of codimension $c$ associated to
  $\mathbb F = \{F_1,\ldots,F_\ell \}$.
  Then the minimal generators of 
  $I_{\mathbb{Z}(\ell,c)}$ are
$$\{F_{j_1}\cdots F_{j_{\ell-c+1}} ~|~ \{j_1,\ldots,j_{\ell-c+1}\} \subseteq [\ell]\}.$$ 
\end{prop}

The next result gives a partial generalization of Lemma \ref{var2}.

\begin{lemma}
\label{lem:generalized_4.2}
Let $\mathbb{Z}(\ell,c) \subseteq \mathbb{P}^n$ be a general star configuration of codimension $c$
associated to $\mathbb F = \{F_1, \ldots, F_\ell\}$ with $d_i = \deg(F_i)$,
$d_1 \leq d_2 \leq \cdots \leq d_\ell$.
If for each variable $x_i \in \{x_0, \ldots, x_n\}$, the product
$F_1 F_2 \cdots F_{\ell-c+1}$ is not squarefree in $x_i$, 
then $I_{\mathbb{Z}(\ell,c)}$ is not geometrically vertex decomposable.
\end{lemma}

\begin{proof}
Fix any monomial order $<$. By Proposition \ref{gens}, 
the product $F_1 F_2 \cdots F_{\ell-c+1}$ is a minimal
generator of $I_{Z(\ell,c)}$ of degree $d = d_1 + d_2 + \cdots + d_{\ell-c+1}$.
Since $d_1 \leq \cdots \leq d_\ell$, the smallest
degree of a minimal generator of
of $I_{Z(\ell,c)}$ is $d$.

By hypothesis, for each variable $x_i$, the polynomial $F_1 \cdots F_{\ell-c+1}$ is not squarefree in $x_i$, that is, $x_i^2$ divides some term of
$F_1 \cdots F_{\ell-c+1}$. 
Consequently, by Lemma \ref{mindeg}, there does not exist a Gr\"{o}bner basis for $I_{\mathbb{Z}(\ell,c)}$  that is squarefree in $x_i$. 

Since this holds for every variable, Lemma \ref{lemma:squarefree} implies that $I_{\mathbb{Z}(\ell,c)}$
does not admit a geometric vertex decomposition with respect to any variable,
and hence $I_{Z(\ell,c)}$ is not geometrically vertex decomposable.
\end{proof}

Similarly, the proof of Theorem \ref{kn l>n+1} can be adapted
to give a condition for the ideal of a generalized star configuration
to fail to be a Knutson ideal.

\begin{thm} \label{not knutson}
    Let $\mathbb{Z}(\ell,c) \subseteq \mathbb{P}^n$ be a 
  general star configuration of codimension $c$ associated to
  $\mathbb F = \{F_1,\ldots,F_\ell \}$.  
  Let $d_i = \deg(F_i)$ and assume that $d_1 \leq d_2 \leq 
  \cdots \leq d_\ell$.    
  Suppose that  $d_1=\cdots =  d_t$ for some $\ell \geq t > \lceil \frac{n+1}{d_1} \rceil$. Then
  $I_{\mathbb{Z}(\ell,c)}$ is not Knutson.
\end{thm}

\begin{proof}
By Proposition \ref{gens},
    it follows that the ideal $I_{\mathbb{Z}(\ell,c)}$ is generated by
    elements of  degrees $d_1+\cdots+d_{\ell-c+1}$ and higher.
    Fix any monomial order $<$.  By the pigeonhole
    principle, there is $F_i$ and $F_j$ among
    $F_1,\ldots,F_t$ that such that the leading
    term of $F_i$ and $F_j$ must have a greatest 
    common divisor bigger than one, or in other
    words, the leading term of $F_iF_j$ is not squarefree.

    We now consider two cases: (1) $ t \geq \ell-c+1$ and
    (2) $t < \ell -c+1$.   
    If $t \geq \ell -c+1$, then 
    $$F_iF_jF_{k_3}\cdots F_{k_\ell-c+1} 
    ~~\mbox{with $\{k_3,\ldots,k_{\ell-c+1}\}
    \subseteq [t] \setminus \{i,j\}$}$$
    is a generator of degree $d_1+\cdots+d_{\ell-c+1}$   whose leading term is divisible by the leading term     of $F_iF_j$, and thus is not squarefree.  We are
    using the fact that the first $t > \ell -c+1$ elements     of $\mathbb{F}$ all have the same degree.

    If $t < \ell -c+1$, then $F_1 \cdots F_i \cdots F_j \cdots
    F_tF_{t+1}\cdots F_{\ell-c+1}$ is a generator of
    smallest degree, whose leading term is divisible by the leading
    term of $F_iF_j$, and thus, the leading term is not 
    squarefree.  

    In both cases, there is a generator of $I_{\mathbb{Z}(\ell,c)}$
    of smallest degree whose initial term is not squarefree.  Since
    this is true for all monomial orders, the initial ideal
    of $I_{\mathbb{Z}(\ell,c)}$ cannot be squarefree, and thus,
    by Theorem \ref{initial ideal}, the ideal $I_{\mathbb{Z}(\ell,c)}$
    cannot be Knutson.
\end{proof}

%%%%%%%%%%%%%%%%%%%%%%%%%%%%%%%%%%%%%%%%%
\subsection*{Acknowledgments}
%%%%%%%%%%%%%%%%%%%%%%%%%%%%%%%%%%%%%%%%%

The results in this paper were inspired
by computer experiments using {\it Macaulay2} \cite{M2} and 
the package {\tt GeometricDecomposabilty} \cite{CVT24}.
Work on this project was
carried out at the Fields Institute in Toronto, Canada
as part of the ``Thematic Program on Commutative Algebra
and Applications'' in 2025.   All of the authors thank
Fields for providing a wonderful environment in which to work and for 
funding to participate in this program. 
We thank Jenna Rajchgot for comments and informing us
about the examples we used in Remark \ref{gvd+knutson}.
Cooper's research is supported by NSERC Discovery Grant 2024-05444.
Guardo was partially supported by Gnsaga of Indam, by 
Progetto Piaceri 2024-26, Università di Catania, and by 
PRIN 2022, “$0$-dimensional schemes, Tensor Theory and 
applications” – funded by the European Union 
Next Generation EU, Mission 4, Component 2 -- CUP: E53D23005670006. 
Marangone gratefully acknowledges that this research was 
supported by the Pacific Institute for the Mathematical Sciences. 
Van Tuyl’s research is supported by NSERC Discovery Grant 2024-05299.
\bibliographystyle{amsplain}
\bibliography{bibfile}

\end{document}